\documentclass[a4paper,10pt]{article}

\usepackage{xcolor}
\usepackage{amsmath}
\usepackage{amssymb} 
\usepackage{amsthm}
\usepackage[latin1]{inputenc}

\newcommand{\R}{\mathbb{R}}
\newcommand{\bq}{\begin{equation}}
\newcommand{\eq}{\end{equation}}
\newcommand{\vo}{\mathrm{dvol}}
\newcommand{\ov}{\overline}
\newcommand{\s}{\mathrm{scal\,}}
\newcommand{\vol}{\mathrm{vol}}

\renewcommand{\d}{\mathrm{d}}

\newtheorem{lem}{Lemma}
\newtheorem{cor}[lem]{Corollary}
\newtheorem{defi}[lem]{Definition}
\newtheorem{thm}[lem]{Theorem}

\theoremstyle{definition}
\newtheorem{ex}[lem]{Example}
\newtheorem{rem}[lem]{Remark}

\title{The Yamabe equation on manifolds of bounded geometry}
\author{Nadine Gro\ss e
\footnote{Faculty of Mathematics, University Leipzig, nadine.grosse@mathematik.uni-leipzig.de}
}
\date{}

\begin{document}
\maketitle

\begin{abstract}
We study the Yamabe problem on open manifolds of bounded geometry and show that under suitable assumptions there exist Yamabe metrics, i.e. conformal metrics of constant scalar curvature. For that, we use weighted Sobolev embeddings.
\end{abstract}

\section{Introduction}

In 1960 Yamabe considered the following problem that became famous as the Yamabe problem:\\

{\it Let $(M,g)$ be a closed Riemannian manifold of dimension $n\geq 3$. Does there exist a Riemannian metric $\ov{g}$ conformal to $g$ that has constant scalar curvature?}\\

This was answered affirmatively by Aubin \cite{Au98}, Schoen \cite{Sch84} and Trudinger \cite{Tr}.

The question can be reformulated in terms of positive solutions of the nonlinear elliptic differential equation:
\bq c u^{p_{crit}-1}=L_{g}u\label{el}, \quad\Vert u\Vert_{p_{crit}}=1,\eq
where $c$ is a constant, $L_{g}=a_n\Delta_{g}+\s_{g}$ with $a_n=4\frac{n-1}{n-2}$ is the conformal Laplacian and $\s_g$ the scalar curvature. We denote $\Vert u\Vert_p:=\Vert u\Vert_{L^p(g)}$ and set $p_{crit}=\frac{2n}{n-2}$. In the following we will omit the index referring to the metric, e.g. $L=L_g$.

If a positive solution $u$ exists, then the conformal metric $\ov{g}=u^\frac{4}{n-2}g$ has constant scalar curvature. Moreover, solutions of \eqref{el} can be characterized as critical points of the Yamabe functional 
\[Q_g(v)=\frac{\int_M vL_gv\vo_g}{\Vert v\Vert_{p_{crit}}^2}.\]
The infimum of the Yamabe functional $Q(M,g)=\inf \{ Q_g(v)\ |\ v\in C_c^\infty(M)\setminus\{0\}\}$ is called the \emph{Yamabe invariant} of $(M,g)$, where $C_c^\infty(M)$ denotes the set of compactly supported real valued functions on $M$. We note that $Q(M,g)$ is a conformal invariant \cite{Schoe}, i.e. for all $g, g'\in [g]=\{ \ov{g}= f^2g\ |\ f\in C_{>0}^\infty(M)\}$ we have $Q(M,g)=Q(M,g')$.\\

Since we take the infimum over all functions with compact support, the definition of the Yamabe invariant can also be used for noncompact manifolds.

What is often referred to as the noncompact Yamabe problem is the question:  Let $(M,g)$ be a Riemannian manifold of dimension $n\geq 3$. Does there exist a complete metric $\ov{g}$ conformal to $g$ that has constant scalar curvature that equals $Q(g)$?

The simplest counterexample is the standard Euclidean space since $Q(\R^n, g_E)>0$. In \cite{Zhir} it was shown that by deleting finitely many points of a closed manifold one can always construct such counterexamples.

Another way to consider a noncompact version of the Yamabe problem is to ask for a positive solution $u\in {H_1^2}\cap L^p$ of \eqref{el} on a noncompact complete manifold that minimizes the Yamabe functional. Here, $H_1^2=H_1^2(g)$ is the completion of $C_c^\infty(M)$ with respect to the norm $\Vert v\Vert_{H_1^2(g)}:=\Vert \d v\Vert_{L^2(g)}+\Vert v\Vert_{L^2(g)}$. The corresponding conformal metric $u^{\frac{4}{n-2}}g$ will have constant scalar curvature but will be in general not complete.\\

In this paper, we want to examine the existence of solutions of the Euler-Lagrange equation that minimize the Yamabe functional, i.e. we consider the second version of the noncompact Yamabe problem described above.\\
In \cite{Kim96}, this problem was studied for positive scalar curvature. In the proof, Aubin's inequality is used which was proofed in \cite[Thm. 9]{Au} for closed manifolds. Unfortunately, this inequality is not true for an arbitrary open manifold, but the proof of Aubin's inequality on closed manifolds carries over to  manifolds with bounded geometry. Recall that a Riemannian manifold $(M,g)$ is of bounded geometry if $g$ is complete and the curvature tensor and all its covariant derivatives are bounded. Thus, in the assumptions of \cite[Thm. 1]{Kim96} bounded geometry should be inserted to make the proof work.\\

In the following, we want to extend this result by relaxing the assumptions on the scalar curvature. Instead of assuming positive scalar curvature, we will assume that $\mu (M,g)$, the infimum of the $L^2$-spectrum of the conformal Laplacian w.r.t. the complete metric $g$ is positive,
i.e. 
\[ \mu (M,g)= \inf\left\{ \int_M vLv\,\vo_g\ \Big|\ v\in C_c^\infty(M),\ \Vert v\Vert_2=1 \right\}>0.\]  
 
\begin{thm}\label{Yam}
Let $(M^n,g)$ be a connected Riemannian manifold of bounded geometry with $\ov{{Q(M, g)}}>Q(M,g)$. Moreover, let $\mu(M,g)>0$. Then, there is a smooth positive solution $v\in H_1^2\cap L^\infty$ of the Euler-Lagrange equation $Lv=Q(M)v^{p_{crit}-1}$ with $\Vert v\Vert_{p_{crit}}=1$.
\end{thm}

Here, $\ov{{Q}}$ denotes the Yamabe invariant at infinity, cf. Definition \ref{Yainf}. Note, moreover, that $\mu>0$ implies $Q>0$, see Lemma \ref{class}.\\
Our method to prove this theorem will be different to the one in \cite{Kim96}, where the noncompact manifold is exhausted by compact subsets. Then the solutions of the corresponding problem on these subsets form a sequence, and it is shown that under suitable assumptions this sequence converges to a global solution.\\
 We will use instead weighted Sobolev embeddings and, therefore, consider a weighted Yamabe problem:

\begin{defi}\label{wsY} Let $\rho$ be a radial admissible weight (cf. \cite[Def. 2]{Skrz}) with $0<\rho\leq 1$. The \emph{weighted subcritical Yamabe constant} of $(M^n,g)$ is defined as
\[ Q^\alpha_p(M,g)=\inf\left\{ \int_M vLv\,\vo_g\ \Big|\ v\in C_c^\infty(M),\ \Vert \rho^\alpha v\Vert_p=1 \right\} \]
where $\alpha\geq 0$ and $p\in[2,p_{crit})$, $p_{crit}=\frac{2n}{n-2}$. If $\alpha=0$, we simply write $Q_p$. 
\end{defi}

For our purpose, it will be sufficient to think of $\rho$ as the radial weight $e^{-r}$ where $r$ is smooth and near to the distance to a fixed point $z\in M$, cf. the Appendix \ref{app} Remark \ref{ch_weight}.

Note that $Q=Q^{\alpha=0}_{p=p_{crit}}$.\\

In Theorem \ref{isoYam}, we will show that for almost homogeneous manifolds (for the Definition see \ref{isoYam}) with uniformly positive scalar curvature one can drop the assumption on $\ov{{Q}}$. This was shown to the author by Akutagawa who proved this by exhaustion of the manifold at infinity, similarly as in \cite[Thm. C]{Aku}. Similar methods are used in \cite[Thm. 1.2]{A} where Akutagawa compares the Yamabe constant of a manifold $M$ with the Yamabe constant on an infinite covering of $M$.

Then, as an application we will apply this result in Example \ref{ex_ex} to products of spheres with hyperbolic spaces that are the noncompact model spaces that appear in the surgery results for the Yamabe invariant in \cite{ADH}.\\

In this paper, we will proceed as follows: In Section 2, we shortly give some general results and the definition of the Yamabe invariant at infinity. Everything that is needed on (weighted) Sobolev embeddings can be found in Appendix A. In Section 3, we will prove Theorem \ref{Yam} by considering a weighted subcritical problem.

The methods developped in this paper to prove existence of solutions of the Yamabe problem on manifold with bounded geometry were adapted to prove similar results for a spinorial Yamabe-type problem for the Dirac operator. That was done in \cite{NG12}.

\textbf{Acknowlegdement.} The author thanks Kazuo Akutagawa for giving many insights to the solutions of the Yamabe problem on noncompact manifolds and showing Theorem \ref{isoYam} which we reproved here by our method. Furthermore, I want to thank Bernd Ammann for many enlightening discussions and hints on weighted Sobolev embeddings.

\section{Preliminaries}

In the rest of the paper, let $(M,g)$ be an $n$-dimensional complete connected Riemannian manifold. In this section we focus on the Yamabe invariant and the Yamabe invariant at infinity. For statements on embeddings, especially on weighted Sobolev embeddings, we refer to Appendix \ref{app}.\\

In the following theorem, we will first collect some basic properties for the Yamabe invariant on manifolds (here not necessarily compact or complete but always without boundary) which we will need in the following, cf. \cite{Schoe}.

\begin{thm}\label{subset}
Let $\Omega_1\subset \Omega_2\subset M$ be open subsets of the Riemannian manifold $(M,g)$ equipped with the induced metric. Then $Q(\Omega_1, g)\geq Q(\Omega_2, g)\geq Q(M,g)$. Moreover,
\[Q(M,g)\leq Q(S^n, g_{st})=n(n-1) \omega_n^{\frac{2}{n}}\]
where $\omega_n$ is the volume of the standard sphere $(S^n, g_{st})$.\\
For any open subset $\Omega \subset S^n$ of the standard sphere, it is $Q(\Omega, g_{st})=Q(S^n, g_{st})$. In particular, the Yamabe invariants of the standard Euclidean and hyperbolic space coincide with the one of the standard sphere. 
\end{thm}

In the sequel, we will left out the metric in the notation of $Q$ if it is clear from the  context to which metric we refer to, e.g. in case of the standard sphere we just write $Q(S^n)$.\\
We further need the Yamabe constant at infinity.

\begin{defi}\label{Yainf}  (see \cite{Kim}) Let $z\in M$ be a fixed point. We denote by $B_R\subset M$ the ball around $z$ w.r.t. the metric $g$  with radius $R$. Then,
\[ \ov{Q(M,g)}:=\lim_{R\to \infty} Q(M\setminus B_R, g). \]
\end{defi}

The limit always exists since with Theorem \ref{subset} we have $Q(M\setminus B_{R_1},g)\leq  Q(M\setminus B_{R_2},g)\leq Q(S^n, g_{st})$ for $R_1\leq R_2$. Hence, $\ov{Q(M)}\geq Q(M)$. Moreover, the definition is independent of the point $z$.

\section{Solution of the Euler-Lagrange equation}

The main aim of this section is to prove Theorem \ref{Yam}. For that, we start by considering the weighted subcritical problem. Firstly, we will prove the existence of solutions of this weighted subcritical problem, i.e. solutions to the corresponding Euler-Lagrange equation, see Lemma \ref{weisub}. Then, the convergence of these solutions will be achieved in two steps: At first, we fix the weight $\rho^\alpha$ and let the subcritical exponent ($p<p_{crit}$) converge to the critical one, cf. Lemma \ref{conv}. Secondly, in Lemma \ref{Qinfmass} we let $\alpha\to 0$, i.e. we establish the convergence to the unweighted critical problem. \\

We start by considering a weighted subcritical problem, see Definition \ref{wsY}, i.e. $2\leq p<p_{crit}$ and $\alpha>0$. That means we look for a solution of the Euler-Lagrange equation  \[Lv=Q^\alpha_p \rho^{\alpha p}v^{p-1} \mathrm{\ where\ }\Vert \rho^{\alpha} v\Vert_p=1.\]
Before considering this problem, we shortly give some preliminaries on the positivity of $Q^\alpha_p$:
 
\begin{lem}\label{ine}\hfill Let $2\leq p\leq p_{crit}$.
\textbf{i)}\ \ For $0\leq \alpha\leq \beta$ and $Q\geq 0$, we have $Q^\alpha_p\leq Q^\beta_p$ and $\lim_{\alpha\to 0} Q^\alpha_p=Q_p$.\\
\textbf{ii)}\ $Q^\alpha_p\geq \limsup_{s\to p}Q^\alpha_s$ for all $\alpha>0$.
\end{lem}

\begin{proof}\hfill\\
i)\ Since $0<\rho\leq 1$ and $\alpha\leq \beta$, $\Vert \rho^\alpha v\Vert_p\geq \Vert \rho^\beta v\Vert_p$. With $Q\geq 0$ we know $\int_M vLv\vo_g\geq 0$ for all $v\in C_c^\infty(M)$. Hence, $Q^\alpha_p\leq Q^\beta_p$ and 
\[\lim_{\alpha\searrow 0} Q^\alpha_p=\inf_{\alpha\geq 0} \inf_v \frac{\int_M vLv\,\vo_g}{\Vert \rho^\alpha v\Vert_p^2}=\inf_v \inf_{\alpha\geq 0} \frac{\int_M vLv\,\vo_g}{\Vert \rho^\alpha v\Vert_p^2}=\inf_v \frac{\int_M vLv\,\vo_g}{\Vert  v\Vert_p^2}=Q_p\]
where $\inf_v$ always goes over all $v\in C_c^\infty(M)\setminus\{0\}$. \\
ii) $\Vert v\Vert_s\to \Vert v\Vert_p$  as $s\to p$ and, thus, we have  
\[ Q^\alpha_p=\inf_v \frac {\int_M vLv\vo_g}{\Vert \rho^\alpha v\Vert_p^2}= \inf_v \lim_{s\to p}\frac {\int_M vLv\vo_g}{\Vert \rho^\alpha v\Vert_s^2} \geq \limsup_{s\to p} \inf_v \frac {\int_M vLv\vo_g}{\Vert \rho^\alpha v\Vert_s^2} = \limsup_{s\to p} Q^\alpha_s\]
where $\inf_v$ is understood as above in i).
 \end{proof}

\begin{rem}\hfill\\
\textbf{i)}\ \ On closed manifolds, if $Q_p\geq 0$, there is already equality in Lemma \ref{ine}.ii, cf. \cite[Lem. V.2.3]{SY}. But for the Euclidean space $(\R^n, g_E)$ we have $Q(\R^n)=Q(S^n)>0$ and $Q_s(\R^n)=0$ for $s\in[2,p_{crit})$, which can be seen when rescaling a radial test function $v(r)\in C_c^\infty(\R^n)$ by a constant $\lambda>0$: $\ov{v}(r)=v(\lambda r)$.\\
\textbf{ii)} On closed Riemannian manifolds, the signs of the Yamabe invariant $Q$ and the first eigenvalue $\mu$ of the conformal Laplacian always coincide. On open manifolds, this is again already false for the Euclidean space where $\mu(\R^n)=0$ but $Q(\R^n)=Q(S^n)$. 
\end{rem}


\begin{lem}\label{class} We have $\mu<0$ if and only if $Q<0$.\\
If we assume additionally that the embedding $H_1^2\hookrightarrow L^{p}$ for $2\leq p\leq p_{crit}$ is continuous, that the scalar curvature is bounded from below and that $\mu >0$, then $Q_p>0$ and $\liminf_{p\to p_{crit}} Q_p>0$.
\end{lem}

\begin{proof}
If $\mu<0$, there exists a function $v\in C_c^\infty(M)$ with $\int_M vLv\,\vo_g<0$. Thus, $Q_p<0$ for all $p$ (in particular $Q=Q_{p_{crit}}<0$). The converse is obtained analogously.\\
This implies that $\mu\geq 0$ if and only if $Q_p\geq 0$ for all $p$. Now let there be a continuous Sobolev embedding, let $\s$ be bounded from below and let $Q_p=0$: We show by contradiction that $\mu=0$, i.e. we argue against the assumption $\mu >0$.
Let $v_i\in C_c^\infty(M)$ be a minimizing sequence: $\Vert v_i\Vert_p=1$ with $ \int v_iLv_i\vo_g \searrow 0$. Then, since $\mu >0$, $\Vert v_i\Vert_2\to 0$. Hence, with the lower bound for the scalar curvature and
\begin{align*}
0\leftarrow \int_Mv_iLv_i\vo_g&=a_n\Vert \d v_i\Vert_2^2+\int_M\s v_i^2\vo_g\\
& \geq a_n\Vert \d v_i\Vert_2^2+\inf_M \s \Vert v_i\Vert_2^2
\end{align*}
  we get $\Vert \d v_i\Vert_2\to 0$. Thus, $v_i\to 0$ in $H_1^2$, but the continuous Sobolev embedding gives $1=\Vert v_i\Vert_p\leq C \Vert v_i\Vert_{H_1^2}$ which is a contradiction.\\
Analogously, we proceed to prove $\liminf_{p\to p_{crit}} Q_p>0$ by contradiction: Let there be a minimizing sequence $v_p\in C_c^\infty(M)$ for $\liminf_{p\to p_{crit}} Q_p=0$, i.e. $\Vert v_p\Vert_p=1$ and $\int_M v_pLv_p\vo_g\to 0$ for $p\to p_{crit}$. This implies, exactly as before, that $\Vert v_p\Vert_{H_1^2}\to 0$. But from the Sobolev embeddings, see Theorem \ref{embcon}, we get 
\[1=\Vert v_p\Vert_p\leq C(p)\Vert v_p\Vert_{H_1^2}\leq \max_{p\in [2, p_{crit}]}C(p) \Vert v_p\Vert_{H_1^2}.\] 
Since each $p\in [2, p_{crit}]$ can be written as $\frac{1}{p}=\frac{1-\theta}{2}+\frac{\theta}{p_{crit}}$ with $0\leq \theta\leq 1$, we then get by interpolation that for all $u\in H_1^2$
\begin{align*}
\Vert u\Vert_p\leq  \Vert u\Vert_2^{1-\theta}\Vert u\Vert_{p_{crit}}^\theta\leq C(2)^{1-\theta}C(p_{crit})^\theta\Vert u\Vert_{H_1^2}.
\end{align*}
Thus, $C(p)\leq C(2)^{1-\theta}C(p_{crit})^\theta$ which implies that $\max_{p\in [2, p_{crit}]} C(p)$ is finite. This provides a contradiction to $\liminf_{p\to p_{crit}} Q_p=0$ (the same interpolation argument applied to $p\in [p-\epsilon, p+\epsilon]$ even shows that $C(p)$ is continuous in $p$). 
\end{proof}

\begin{rem}For closed manifolds and $Q\geq 0$, it holds $ Q(M,g)=\inf_{\ov{g}\in [g]} \mu(\ov{g})\vol(\ov{g})^\frac{2}{n}$ where $[g]$ denotes the conformal class of $g$ and $\vol(\ov{g})$ is the volume of $(M,\ov{g})$. For complete manifolds and $Q\geq 0$, we have analogously that
\[ Q(M,g)=\inf_{\ov{g}\in [g],\ \vol(\ov{g})<\infty} \mu(\ov{g})\vol(\ov{g})^\frac{2}{n}.\]\\
For manifolds of finite volume, this implies that from $\mu=\mu (g)=0$ we obtain $Q=0$. 
\end{rem}


Now, we come to solutions of the weighted subcritical problem.

\begin{lem}\label{weisub}
Assume that the embedding $H_1^2\hookrightarrow \rho^\alpha L^p$ is compact for all $\alpha >0$ and $2\leq p<p_{crit}= \frac{2n}{n-2}$. Furthermore, let $\tilde{c}\geq\s\geq c$ for constants $\tilde{c}$ and $c$. Let $\mu >0$.\\
Then, for any $\alpha>0$ and $2\leq p<p_{crit}$, there exists a positive function $v\in C^\infty \cap H_1^2$ with $Lv=Q^\alpha_p \rho^{\alpha p} v^{p-1}$ and $\Vert \rho^\alpha v\Vert_p=1$.
\end{lem}

\begin{proof} Firstly, from Lemma \ref{class} we know that $Q> 0$ and, thus, by Lemma \ref{ine}.i $Q_p^\alpha> 0$ for all $\alpha>0$. Let now $\alpha>0$ and $2\leq p<p_{crit}$ be fixed. Moreover, let $v_i\in C_c^\infty(M)$ be a minimizing sequence for $Q_p^\alpha$, i.e. $\int_M v_i Lv_i\,\vo_g\searrow Q^\alpha_p$ and $\Vert \rho^\alpha v_i\Vert_p=1$. Without loss of generality, we can assume that $v_i$ is nonnegative. Moreover, with 
\[ 0 \leq Q_p^\alpha \swarrow \int_M v_iLv_i\vo_g=a_n\Vert \d v_i\Vert_2^2+\int_M \s\ v_i^2\,\vo_g \geq \mu \Vert v_i\Vert_2^2\]
for $i\to \infty$ and $\mu>0$, we obtain that $\Vert v_i\Vert_2$ is uniformly bounded. Hence, using that $\int_M v_iLv_i\vo_g  \geq a_n\Vert dv_i\Vert_2^2 + c\Vert v_i\Vert_2^2$ the sequence $v_i$ is uniformly bounded in $H_1^2$. So, $v_i\to v\geq 0 $ weakly in $H_1^2$ with $\Vert \d v\Vert_2\leq \liminf \Vert \d v_i\Vert_2$. Due to the compactness of the Sobolev embeddings in Theorem \ref{embcon}, $\rho^\beta v_i$ converges to $\rho^\beta v$ even strongly both in $L^p$ and in $L^2$ for all $\beta>0$. In particular, for $\beta =\alpha $ we obtain $\Vert \rho^\alpha v\Vert_p=1$.

Moreover, for any $w\in L^2, w\geq 0$ we have $\rho^\beta w\nearrow w$ pointwise as $\beta\to 0$ and, thus, with the theorem of dominated convergence $\Vert (\rho^\beta-1)w\Vert_2\to 0$ as $\beta \to 0$. Since the scalar curvature is bounded, we further get  $\int_M \s \rho^{2\beta} w^2\, \vo_g \to 
\int_M \s w^2\, \vo_g$ as $\beta\to 0$. Hence, for every $\epsilon >0$ we have for $i$ large enough that
\[ \int_M \s v^2\vo_g \xleftarrow{\beta\to 0} \int_M \s \rho^{2\beta} v^2\leq  \int_M \s \rho^{2\beta}v_i^2\, \vo_g +\epsilon \xrightarrow{\beta\to 0} \int_M \s v_i^2 \vo_g +\epsilon .\]
Thus, 
\begin{align*} \int_M vLv\, \vo_g &= a_n \Vert \d v\Vert_2^2+\int_M \s v^2\,\vo_g \leq a_n \liminf_{i\to\infty} \Vert \d v_i\Vert_2^2+ \liminf_{i\to\infty} \int_M \s  v_i^2\, \vo_g\\
&\leq \lim_{i\to \infty} \int_M v_i Lv_i\, \vo_g=Q^\alpha_p.
\end{align*}
Hence, $\Vert \rho^\alpha v\Vert_p^{-2}\int_M vLv\vo_g\leq Q^\alpha_p$. But since $Q^\alpha_p$ is the minimum, it already holds equality and $v$ fulfills the Euler-Lagrange equation $Lv=Q^\alpha_p \rho^{\alpha p}v^{p-1}$ with $ \Vert \rho^\alpha v\Vert_p=1$.

Furthermore, since $Q^\alpha_p>0$ and $\s$ is bounded, there is a constant $C > 0$ with $\Delta v + Cv\geq 0$. Thus, due to the maximum principle, $v$ is everywhere positive. From local elliptic regularity theory, we know that $v$ is smooth.
\end{proof}


Before considering the convergence of solutions, we observe that 

\begin{rem}\label{alpha0}\hfill\\
\textbf{i)} From Lemma \ref{ine}.i it follows:\\
Let $Q(\R^n, g_E)>Q(M)$. Then, there exists an $\alpha_0>0$ such that for all $0\leq\alpha\leq\alpha_0$  we have $Q(\R^n)>Q^\alpha_{p_{crit}}(M).$ \\
\textbf{ii)} In the subsequent, we will often make use of the following without any further reference: \\
If $v\in L^2$ is a weak smooth solution of $Lv=c\rho^{\alpha p} v^{p-1}$ with $0< \Vert \rho^\alpha v\Vert_p\leq 1$ and bounded scalar curvature. Then, $v\in H_1^2$ and, hence, it is an admissible test function for $Q_p^\alpha$, i.e. $Q_p^\alpha(M,g)\leq \Vert \rho^\alpha v\Vert_p^{-2}\int_M vL_gv\vo_g$.
This can be seen immediately since both integrals $\int_M \s v^2\vo_g$ and $\int_M vLv\vo_g = \int_M \rho^{\alpha p}v^p\vo_g$  exist and are finite which implies the same for $\int_M v\Delta v\vo_g$. By a cut-off function argument and $v\in L^2$, one sees that $\int_M v\Delta v\vo_g=\int_M|\d v|^2\vo_g$. Thus, $v\in H_1^2$.
\end{rem}

Next, we show that a suitable subsequence of the weighted subcritical solutions given in Lemma \ref{weisub} converges to a solution of the weighted critical problem, i.e. we fix the weight $\alpha$ and let the exponent converge, i.e. $p\to p_{crit}$:

\begin{lem}\label{conv}
Let $v_{\alpha,p}\in H_1^2$ ($\alpha >0$, $p<p_{crit}$) be smooth positive solutions of  $Lv_{\alpha, p}=Q^\alpha_p\rho^{\alpha p} v_{\alpha,p}^{p-1}$ with $\Vert \rho^\alpha v_{\alpha, p}\Vert_p=1$. We assume that $Q(\R^n, g_E)>Q(M)$. Furthermore, let $M$ have bounded geometry and $\mu>0$. Let $\alpha<\alpha_0$ be fixed where $\alpha_0$ is chosen as in Remark \ref{alpha0}.\\
Then, abbreviating $v_p=v_{\alpha, p}$ 
\begin{itemize} 
  \item[a)] there exists $k>0$ such that $\sup v_p\leq k$ for all $p$
  \item[b)] for $p\to p_{crit}$, $v_p\to v_\alpha\geq 0$ in the $C^2$-topology on each compact set, $v_\alpha \in H_1^2\cap L^\infty$ and \[ Lv_\alpha = Q^\alpha_{p_{crit}}\rho^{\alpha p_{crit}} v_\alpha^{p_{crit}-1}\ \mathrm{with}\ \Vert \rho^\alpha v_\alpha\Vert_{p_{crit}}=1.\]
 \end{itemize}
\end{lem}
 
\begin{proof} From Lemma \ref{max} in the Appendix we know that each $v_p$ has a maximum.\\
a) Let $x_p\in M$ be a point where $v_p$ attains its maximum. We prove the claim by contradiction and assume that $m_p:=v_p(x_p)\to \infty$.\\
If, for $p\to p_{crit}$, the sequence $x_p$ converges to a point $x\in M$, we could simply use Schoen's argument \cite[pp. 204-206]{SY} and introduce geodesic normal coordinates around $x$ to show that $m_p$ is bounded from above by a constant independent of $p$.\\
In general, the sequence $x_p$ can escape to infinity, that is why we take a geodesic normal coordinate system around each $x_p$ with radius $\epsilon<\text{inj} (M)=$ the injectivity radius of $M$. This coordinate system will be denoted by $\phi_p$ and $\phi_p: B_\epsilon(0)\subset \R^n\to M$ with $\phi_p(0)=x_p$. The bounded geometry of $M$ and the boundedness  of each $v_p$  ensures that Schoen's argument can be adapted:\\
With respect to the geodesics coordinates introduced above, we have the following expansions \cite[pp. 60-61]{LP} 
\begin{align*} g_{rq}^p(x)&=\delta_{rq}+\frac{1}{3} R^p_{rijq}x^ix^j+ O(|x|^3)\\
 \det g_{rq}^p(x)&=1-\frac{1}{3} R^p_{ij}x^ix^j+O(|x|^3), 
 \end{align*}
where the upper index $p$ always refers to the coordinate system $\phi_p$ around $x_p$, $R^p_{rijq}$ denotes the Riemannian curvature in $x_p$ and $R^p_{ij}$ the Ricci curvature in $x_p$. 
After rescaling $u_p ={m_p}^{-1} v_p (\phi_p(\delta_p x))$ with $\delta_p =m_p^{(2-p)/2}\to 0$ (note that $\delta_p\to 0$ as $p\to p_{crit}$) we have $u_p: B_{\frac{\epsilon}{\delta_p}}(0)\to M$ with $u_p(0)=1$, $u_p\leq 1$. The weight function in the new coordinates will be denoted by ${\rho_p}(x):=\rho(\phi_p(\delta_p x))$.\\
In the following, we identify $\phi_p(\delta_p x)$ with $\delta_p x$  and omit $\phi_p$ in the notation. \\
The Euler-Lagrange equation in the geodesic coordinates reads (compare \cite{SY})
\bq \frac{1}{b_p}\partial_j(b_p a_p^{ij}\partial_i u_p)- c_p u_p +Q^{\alpha}_p \rho_p^{\alpha p} u_p^{p-1}=0\label{EL-Rn}\eq
where
\begin{align}\label{coeff}
  a_p^{ij}(x)&=a_n g^{ij}(\delta_p x)\to a_n\notag\\
  b_p (x)&=\sqrt{\det\, g(\delta_p x)}\to 1\\
  c_p (x)&=m_p^{1-p} \s (\delta_p x)\to 0\notag
\end{align}
for $p\to p_{crit}$. The convergences in \eqref{coeff} are $C^1$ on any compact subset of $\R^n$.\\
Now, we can follow the proof of Schoen, and we show with interior Schauder and global $L^p$ estimates that $u_p$ is bounded in $C^{2,\gamma}$ (for appropriate $\gamma$) on each compact subset $K$ and, thus,  obtain $u_p\to u$ in $C^2$ on $K$: We have on a compact subset $K\subset \Omega \subset \R^n$ the inner 
$L^p$ estimate (using $\rho_p\leq 1$ and $u_p\leq 1$):  \[\Vert u_p\Vert_{H_2^p(K)}\leq C_{K}(\Vert u_p\Vert_{L^q (\Omega)} +\Vert u_p^{p-1}\Vert_{L^q(\Omega)})\leq 2C_K \vol(\Omega)^{\frac{1}{q}}\leq C(K,\Omega),\] 
where $q$ and $p$ are conjugate and $C(K,\Omega)$ only depends on the subsets $K,\Omega$ and $(M,g)$.\\  
Together with the continuous embedding $H_1^q\hookrightarrow C^{0,\gamma}$ where $\gamma\leq 1-\frac{n}{q}$, we obtain, that $u_p$ and, thus, also $u_p^{p-1}$, are uniformly bounded in $C^{0,\gamma}(K)$ (for possibly smaller $\gamma$). With the interior Schauder estimate 
\[\Vert u_p\Vert_{C^{2,\gamma}(K)}\leq C (\Vert u_p\Vert_{C^0(\Omega)}+\Vert u_p^{p-1}\Vert_{C^{0,\gamma}(\Omega)})\]
$u_p$ is uniformly bounded in $C^{2,\gamma}(K)$. With the theorem of Arcela-Ascoli, we obtain, by going to a subsequence if necessary, that $u_p\to u$ in $C^2$ on each compact subset. Thus, $1\geq u\geq 0$ and $u(0)=1$.\\
Firstly, we argue that $u\in L^{p_{crit}}(\R^n)$: We estimate
\begin{align*}\int_{|x|<\epsilon \delta^{-1}_p}u_p^{p} b_p\,\vo_{g_E}&= \int_{B_\epsilon(x_p)} \delta_p^{\frac{2p}{p-2}-n} v_p^p\,\vo_{g} \leq C\delta_p^{\frac{2p}{p-2}-n} \Vert v_p\Vert _{H_1^2(M)}^p
\end{align*}
where the equality is obtained by change of variables with $b_p$ as in \eqref{coeff} and the inequality is the Sobolev embedding (see Theorem \ref{embcon}). Using $Lv_p=Q^\alpha_p\rho^{\alpha p} v_p^{p-1}$ with $\Vert\rho^\alpha v_p\Vert_p=1$, we obtain 
\begin{align*}Q^\alpha_p&=\int_M v_p Lv_p\vo_g=a_n\Vert\d v_p\Vert^2_{L^2(M)}+\int_M \s v_p^2\vo_g\\
&\geq a_n\Vert\d v_p\Vert_{L^2(M)}^2+\inf \s \Vert v_p\Vert^2_{L^2(M)}
\end{align*}
and, thus,
\begin{align}\label{h1_a}
\int_{|x|<\epsilon \delta^{-1}_p}u_p^{p} b_p\,\vo_{g_E}
&\leq C\delta_p^{\frac{2p}{p-2}-n} \left(\Vert v_p\Vert_{L^2(M)} + \left( a_n^{-1}\left(Q_p^\alpha-\inf \s \Vert v_p\Vert_{L^2(M)}^2\right)\right)^{\frac{1}{2}}\right)^p.
 \end{align}
From $\mu >0$, we have additionally that 
\begin{align}\label{est_L2}\Vert v_p\Vert_{L^2}^2\leq \mu^{-1}\int v_pLv_p\vo_g=\mu^{-1}Q^\alpha_p.\end{align} With $\limsup_{p\to p_{crit}}Q_p^\alpha\leq Q_{p_{crit}}^\alpha$ (Lemma \ref{ine}.ii), we get that $\Vert v_p\Vert_{L^2}$ is uniformly bounded on $p\in (2,p_{crit})$. Moreover, $\frac{2p}{p-2}-n\searrow 0$ for $p\to p_{crit}$. Hence, with \eqref{h1_a} the integral $\int_{|x|<\epsilon \delta^{-1}_p}u_p^{p} b_p\,\vo_{g_E}$ is bounded from above by a constant independent of $p$. Thus, by the Lemma of Fatou $u\in L^{p_{crit}}(\R^n)$.\\

Now, in order to construct a contradiction, we distinguish between two cases:

At first, we consider the case that $x_p$ escapes to infinity if $p\to p_{crit}$:

Then, $\rho_p\to 0$ as $p\to p_{crit}$. With the  $C^2$-convergence of $u_p\to u$ on compact subsets and \eqref{EL-Rn}, this implies
 \[        
a_n\Delta u=\limsup_{p\to p_{crit}} (Q^{\alpha}_p(M)\rho_p^{\alpha p} u_p^{p-1})= 0\] on $\R^n$. From the maximum principle, $u(0)=1$ and $u\leq 1$, we obtain that $u\equiv 1$ which contradicts $u\in L^{p_{crit}}(\R^n)$.\\

Secondly, we consider the remaining case that a subsequence of $x_p$ converges to a point $y\in M$. With $u\geq0$ and $u(0)=1$, we obtain $u>0$ from the maximum principle.\\
Moreover, with $\Vert \rho^\alpha v\Vert_p=1$ we have for $\epsilon_1\leq \epsilon$
\begin{align*}
\int_{|x|<{\epsilon_1}\delta^{-1}_p}u_p^{p} b_p\,\vo_{g_E}&\leq (\min_{B_{{\epsilon_1}}(x_p)} \rho^{\alpha p})^{-1}\int_{B_{{\epsilon_1}}(x_p)} \rho^{\alpha p}v_p^{p} \delta_p^{\frac{2p}{p-2}-n}\,\vo_g\\ &\leq  (\min_{B_{{\epsilon_1}}(x_p)} \rho^{\alpha p})^{-1}\delta_p^{\frac{2p}{p-2}-n} \to  \max_{B_{{\epsilon_1}}(y)} \rho^{-\alpha p_{crit}}
\end{align*}
for $p < p_{crit}$ and by Fatou's Lemma, we obtain $\Vert u\Vert_{p_{crit},g_E}\leq \max_{B_{{\epsilon_1}}(y)} \rho^{-\alpha }$. Letting $\epsilon_1\to 0$ we have  $\Vert u\Vert_{p_{crit},g_E}\leq \rho^{-\alpha }(y)$. 

From \eqref{EL-Rn}, $u_p\to u$ in $C^2$ on compact subsets and that $\rho_p$ converges to the constant $\rho (y)$, we get \begin{align*}        
a_n\Delta u&=\limsup_{p\to p_{crit}} (Q^{\alpha}_p(M)\rho_p^{\alpha p} u_p^{p-1})\\
&\leq  (\limsup_{p\to p_{crit}} Q^{\alpha}_p(M)) \rho^{\alpha p_{crit}}(y) 
u^{p_{crit}-1} \leq Q^\alpha_{p_{crit}}(M) \rho^{\alpha p_{crit}}(y) 
u^{p_{crit}-1}
\end{align*} on $\R^n$.
Note that $u$ is an admissible test function, i.e. $Q(\R^n)\leq Q_{g_E}(u)$, which can be seen by the following: From $0\leq u\leq 1$, $L_{\R^n}u=cu^{p_{crit}-1}$ and $u\in L^{p_{crit}}(\R^n)$ we get by Lemma \ref{max} that $\lim_{|x|\to \infty } u=0$. By stereographic projection we can pullback everything from $\R^n$ to $S^n\setminus \{z\}$ for a fixed $z\in M$. The pullback of $u$ we call $\hat{u}$. Then, $L_{S^n}\hat{u}=c\hat{u}^{p_{crit}-1}$ on $S^n\setminus\{z\}$ and $ u\in L^{p_{crit}}(S^n\setminus\{z\})$. Using a cut-off argument near $z$, one can remove the singularity and gets $L_{S^n}\hat{u}=c\hat{u}^{p_{crit}-1}$ on $S^n$ which implies by global regularity theory that $\hat{u}\in H_1^2(S^n)$. Hence, by conformal invariance $u$ is also an admissible test function for $Q(\R^n)$ and, thus,

\begin{align*} Q(\R^n)&\leq \frac{\int a_nu\Delta u\,\vo_{g_E}}{\Vert u\Vert_{p_{crit},g_E}^2}\leq Q^\alpha_{p_{crit}}(M) \rho^{\alpha p_{crit}}(y) \Vert u\Vert_{p_{crit},g_E}^{p_{crit}-2}\\
& \leq Q^\alpha_{p_{crit}}(M)  \rho^{\alpha p_{crit}}(y)\rho^{-\alpha (p_{crit}-2)}(y)\leq Q^\alpha_{p_{crit}}(M)\rho^{2\alpha }(y)\\
&\leq Q^\alpha_{p_{crit}}(M),
\end{align*}
which contradicts the assumption that $Q(\R^n, g_E)>Q(M)$ and $\alpha \leq \alpha_0$ (see Remark \ref{alpha0}). Thus, there exists a $k>0$ with $m_p\leq k$.\\
b) From a), we know $\max v_p\leq k$ for all $p$. Thus, we can apply the interior Schauder and inner $L^p$-estimates as above and obtain, that $v_p\to v_\alpha$ in $C^2$ on each compact subset $K$. Moreover, $v_\alpha\in L^\infty$.
Together with \eqref{est_L2} and Lemma \ref{ine}, we get that $v_\alpha\in L^2$ and \[Lv_\alpha=(\limsup_{p\to p_{crit}} Q^{\alpha}_p) \rho^{\alpha p_{crit}}v_\alpha^{p_{crit}-1}\leq Q^\alpha_{p_{crit}} \rho^{\alpha p_{crit}} v_\alpha^{p_{crit}-1}.\] Clearly, by the Lemma of Fatou $\Vert \rho^\alpha v_\alpha\Vert_{p_{crit}}\leq 1$ and smoothness of $v_\alpha$ follows from standard elliptic regularity theory.\\
It remains to show that $\Vert \rho^\alpha v_\alpha\Vert_{p_{crit}}=1$. Firstly, we assume that $v_\alpha=0$:\\
Since \[Q_p\leq \frac{\int_M v_pLv_p\vo_g}{ \left( \int_M v_p^p\vo_g\right)^\frac{2}{p}}= Q^\alpha_p \Vert v_p\Vert_p^{-2}\] and $Q_p>0$, $\liminf_{p\to p_{crit}} Q_p >0$ (Lemma \ref{class}), we have
\[ \limsup_{p\to p_{crit}} \Vert v_p\Vert_p\leq \limsup_{p\to p_{crit}}\left( \frac{Q^\alpha_p}{Q_p}\right)^{\frac{1}{2}}\leq \left( \frac{Q^\alpha_{p_{crit}}}{\liminf_{p\to p_{crit}} Q_p}\right)^{\frac{1}{2}}=:c<\infty.\]
Thus, 
\[ \limsup_{p\to p_{crit}}\int_{M\setminus B_R} \rho^{\alpha p}v_p^p\vo_g\leq \limsup_{p\to p_{crit}}(\max_{M\setminus B_R} \rho^{\alpha p}\ \Vert v_p\Vert_p^p)\leq e^{-(R-\xi)\alpha p_{crit}} c^{{p_{crit}}}\]
where the last inequality follows with Remark \ref{ch_weight}.\\
Choose $R=R(\alpha)$ big enough such that $\limsup_{p\to p_{crit}} \int_{M\setminus B_R} \rho^{\alpha p}v_p^p\vo_g\leq \frac{1}{2}$. Then, with $\Vert \rho^\alpha v_p\Vert_p=1$ we get \[\limsup_{p\to p_{crit}} \int_{B_R} \rho^{\alpha p}v_p^p\vo_g\geq \frac{1}{2},\] which contradicts the assumption that $v_p\to v_\alpha=0$. Thus, $\Vert \rho^\alpha v_\alpha\Vert_{p_{crit}}>0$.\\
Using the smoothness of $v_\alpha\in L^2$ and that it weakly fulfills $Lv_\alpha\leq Q^\alpha_{p_{crit}}\rho^{\alpha p_{crit}} v_\alpha^{p_{crit}-1}$, we can compute \[0 < Q^\alpha_{p_{crit}}\leq \frac{\int_M v_\alpha Lv_\alpha\vo_g}{\left( \int_M \rho^{\alpha p_{crit}} v_\alpha^{p_{crit}}\vo_g\right)^{\frac{2}{p_{crit}}}}
\leq Q^\alpha_{p_{crit}} \Vert \rho^\alpha v_\alpha\Vert_{p_{crit}}^{p_{crit}-2}\]
and obtain $\Vert \rho^\alpha v_\alpha\Vert_{p_{crit}}=1$ and, hence, equality in $Lv_\alpha=Q^\alpha_{p_{crit}}\rho^{\alpha p_{crit}} v_\alpha^{p_{crit}-1}$.\\
In particular, we have $\limsup_{p\to p_{crit}} Q^\alpha_p=Q^\alpha_{p_{crit}}$.
\end{proof}

Similarly, we now take the limit for $\alpha\to 0$:

\begin{lem}\label{Qinfmass}
Let $v_\alpha\in H_1^2\cap L^\infty$ ($\alpha_0\geq \alpha >0$) be smooth and positive solutions of  $Lv_{\alpha}=Q^\alpha_{p_{crit}}\rho^{\alpha p_{crit}} v_{\alpha}^{p_{crit}-1}$ with $\Vert \rho^\alpha v_{\alpha}\Vert_{p_{crit}}=1$. Furthermore, let $M$ have bounded geometry and let  $Q(\R^n, g_E)>Q(M)$.\\
Then, there exists $k>0$ such that $\sup v_\alpha\leq k$ for all $\alpha$. Moreover, for $\alpha\to 0$, $v_\alpha\to v$ in $C^2$-topology on each compact set, $v\in H_1^2\cap L^\infty$ and $Lv= Q_{p_{crit}} v^{p_{crit}-1}$.\\ 
If additionally $\ov{Q(M,g)}>Q(M,g)$, we have $\Vert v\Vert_{p_{crit}}=1$.
\end{lem}

\begin{proof} First note that by Lemma \ref{max}, $\lim_{|x|\to \infty} v_\alpha =0$ where $|x|$ denotes the distance of $x$ to a fixed point $z\in M$. Then, the first part is proven in the same way as in Lemma \ref{conv}: Let $x_\alpha\in M$ be points where $v_\alpha$ attains its maximum $m_\alpha:=v_\alpha(x_\alpha)$. We assume that $m_\alpha\to \infty$. In the same way as in Lemma \ref{conv} we introduce rescaled geodesic coordinates $\phi_\alpha$ on $B_\epsilon(x_\alpha)$ (where $\epsilon$ is smaller than the injectivity radius of $M$) and obtain $u_\alpha=m_\alpha^{-1}v_\alpha(\phi_\alpha(\delta_\alpha x))$ with $\delta_\alpha=m_\alpha^{(2-p_{crit})/2}$ that fulfills the same (after changing the upper index $p$ to $p_{crit}$ and the lower $p$ to $\alpha$) Euler-Lagrange equation \eqref{EL-Rn}. Using interior Schauder and global $L^{p}$-estimates, one can again prove that $u_\alpha\in H_1^{q_{crit}}$ and, thus, uniformly bounded in $C^{0,\gamma}(K)$ for compact subsets $K\subset M$ and appropriate $\gamma$. Hence, $u_\alpha\to u$ in $C^2$ on compact subsets with $u\geq 0$ and $u(0)=1$.\\
An analogous estimate as in Lemma \ref{conv} shows that $\int_{|x|<\epsilon\delta_\alpha^{-1}} u^{p_{crit}}_\alpha b_\alpha\vo_{g_E}$ is bounded (independent on $\alpha$). Thus, the lemma of Fatou gives $u\in L^{p_{crit}}(\R^n)$ and, moreover, 
\[ a_n \Delta u=\limsup_{\alpha\to 0}(Q^\alpha_{p_{crit}} \rho_\alpha^{\alpha p_{crit}}u_\alpha^{p_{crit}-1})\leq Qu^{p_{crit}-1} \limsup_{\alpha\to 0} \max_{B_\epsilon(x_\alpha)} \rho^{\alpha p_{crit}}.\]
With Remark \ref{ch_weight} we get
\begin{align*} a_n \Delta u&\leq Qu^{p_{crit}-1} \limsup_{\alpha\to 0}   \max_{B_\epsilon(x_\alpha)} e^{-\alpha p_{crit}(|x|-\xi)}\leq Qu^{p_{crit}-1} \limsup_{\alpha\to 0} e^{-\alpha p_{crit}(|x_\alpha|-\xi-\epsilon)}\\
&= Qu^{p_{crit}-1} \limsup_{\alpha\to 0} e^{-\alpha p_{crit}|x_\alpha|}.
\end{align*}
In case that $\alpha |x_\alpha|\to \infty$ as $\alpha\to 0$, the last limes goes to zero and this leads to a contradiction as in Lemma \ref{conv} where the case of $x_p$ tending to infinity as $p\to p_{crit}$ was discussed. Thus, from now on we can assume that $\alpha |x_\alpha|$ is bounded.\\
Moreover, we can estimate 
as in Lemma \ref{conv}
that
\[  \int_{|x|<\epsilon\delta_\alpha^{-1}} u^{p_{crit}}_\alpha b_\alpha\vo_{g_E}\leq 
\max_{B_{\epsilon}(x_\alpha)} \rho^{-\alpha p_{crit}}.
\]
and with Remark \ref{ch_weight}  we get 
\[\int_{|x|<\epsilon\delta_\alpha^{-1}} u^{p_{crit}}_\alpha b_\alpha\vo_{g_E}\leq \max_{B_{\epsilon}(x_\alpha)} e^{\alpha p_{crit} (|x|+\xi)}=
e^{\alpha p_{crit} (|x_\alpha|+\epsilon+\xi)}\]
and, hence, $\Vert u\Vert_{p_{crit},g_E}\leq \liminf_{\alpha\to 0} e^{\alpha (|x_\alpha|+\epsilon+\xi)}=\liminf_{\alpha\to 0} e^{\alpha |x_\alpha|}$.\\
Thus, as in Lemma \ref{conv} we get
\begin{align*}
 Q(\R^n)&\leq \frac{\int a_nu\Delta u\,\vo_{g_E}}{\Vert u\Vert_{p_{crit},g_E}^2}\leq Q(M) \limsup_{\alpha\to 0} e^{-\alpha |x_\alpha|p_{crit}} \Vert u\Vert_{p_{crit},g_E}^{p_{crit}-2}\\
& \leq Q(M) \limsup_{\alpha\to 0} e^{-\alpha |x_\alpha|p_{crit}}\liminf_{\alpha\to 0}e^{\alpha |x_\alpha|(p_{crit}-2)}\leq  c\, Q(M)
\end{align*}
where $c\geq 1$ and the last inequality follows since the  both limits $\limsup_{\alpha\to 0} e^{-\alpha |x_\alpha|p_{crit}}$  and $\liminf_{\alpha\to 0}\ e^{\alpha |x_\alpha|(p_{crit}-2)}$ are finite and $\geq 1$ since we assumed that $\alpha |x_\alpha|$ is bounded. But this gives a contradiction to $Q(\R^n)>Q(M)$. Hence, $v_\alpha$ has to be bounded uniformly in $\alpha$.\\
Then we can again use interior Schauder and inner $L^p$ estimates and obtain $v_\alpha\to v$ in $C^2$ on compact subsets with $Lv=Qv^{p_{crit}-1}$. Moreover, as before we obtain from \eqref{h1_a} that $v_\alpha$ are uniformly bounded in $L^2$ and, hence, $v\in L^2$.\\ 
Assume now that $\ov{Q(M,g)}>Q(M,g)\geq 0$. Clearly, also $\rho^\alpha v_\alpha\to v$ in $C^2$ on compact subsets, $\Vert v\Vert_{p_{crit}}\leq 1$ and smoothness of $v$ follows again from elliptic regularity theory. We have to show that $\Vert v\Vert_{p_{crit}}=1$.\\
Firstly assume that $v_\alpha\to v\equiv 0$. Then, for a fixed ball $B_r:=B_r(z)$ around $z\in M$ with radius $r$ we get that
\begin{align*}
Q(M)&=\liminf_{\alpha\to 0} Q^\alpha_{p_{crit}}(M)= \liminf_{\alpha\to 0} \int_M v_\alpha Lv_\alpha\,\vo_g\\
& \geq \liminf_{\alpha\to 0} \int_{M\setminus B_r} v_\alpha Lv_\alpha\,\vo_g +\liminf_{\alpha\to 0} \int_{B_r} v_\alpha Lv_\alpha\,\vo_g,
\end{align*}
where the first equality is given by Lemma \ref{ine}.i and the second equality follows from $Lv_\alpha = Q^\alpha_{p_{crit}}\rho^{\alpha p_{crit}} v_\alpha^{p_{crit}-1}$ and $ \Vert \rho^\alpha v_\alpha\Vert_{p_{crit}}=1$. The last summand vanishes as $\alpha\to 0$. In order to estimate the other summand, we introduce a smooth cut-off function $\eta_r\leq 1$ with support in $M\setminus B_r$ and $\eta_r\equiv 1$ on $M\setminus B_{2r}$. Then, for $\alpha\to 0$
\begin{align*}
&\left|\int_{M\setminus B_r}\!\eta_r v_\alpha L(\eta_r v_\alpha)\,
\vo_g-\int_{M\setminus B_r}\!\! v_\alpha Lv_\alpha\,\vo_g \right|\\
\phantom{hh}&= \left| \int_{B_{2r}\setminus B_r}\!\!\eta_r v_\alpha L(\eta_r v_\alpha)\,\vo_g -\int_{B_{2r}\setminus B_r}\!\! v_\alpha Lv_\alpha\,\vo_g  \right| \to 0.
\end{align*}
since $v_\alpha\to 0$ in $C^2$ on each compact set. Hence, with $\int_M v_\alpha^{p_{crit}}\vo_g\geq \int_M (\rho^{\alpha} v_\alpha)^{p_{crit}}\vo_g=1$ and Lemma \ref{ine}.i we obtain
\begin{align*}
 Q(M)&= \liminf_{\alpha\to 0} Q^\alpha_{p_{crit}}(M) \geq \liminf_{\alpha\to 0} \int_{M\setminus B_r}\!\!\eta_rv_\alpha L(\eta_r v_\alpha)\,\vo_g\\
 &\geq \liminf_{\alpha\to 0} Q^\alpha_{p_{crit}}(M\setminus  B_r) \left( \int_{M\setminus B_r}\!\! (\eta_r v_\alpha)^{p_{crit}}\,\vo_g\right)^{\frac{2}{{p_{crit}}}}\\
 &=  \liminf_{\alpha\to 0} Q^\alpha_{p_{crit}}(M\setminus  B_r) \left( \int_M v_\alpha^{p_{crit}}\vo_g- \int_{B_{2r}}\!\! (1-\eta_r^{p_{crit}})v_\alpha^{p_{crit}}\,\vo_g\right)^{\frac{2}{p_{crit}}}\\
 &\geq Q(M\setminus B_r),
\end{align*}
where the integral over $B_{2r}$ vanishes again since $v_\alpha\to 0$ on compact sets.\\
Thus, $\ov{Q(M)}\leq Q(M)$ which contradicts the assumption. Thus, we have $\Vert v\Vert_{p_{crit}}>0$.\\
Since $Lv= Q(M) v^{p_{crit}-1}$, $\Vert v\Vert_{p_{crit}}\leq 1$ and $v\in L^2$, we further obtain  that
\[ Q(M)\leq\frac{\int_M vLv\,\vo_g}{\Vert v\Vert_{p_{crit}}^2}=Q(M)\Vert v\Vert_{p_{crit}}^{p_{crit}-2}\leq Q(M),\]
i.e. $\Vert v\Vert_{p_{crit}}=1$.
\end{proof}

\begin{proof}[Proof of Theorem \ref{Yam}]
Combining Lemma \ref{weisub} and \ref{Qinfmass} with \cite[Cor. 2]{Skrz} (cf. Appendix \ref{app} Theorem \ref{embcon}) where the required Sobolev embeddings are proven for manifolds of bounded geometry, we obtain Theorem \ref{Yam}. 
\end{proof}
For almost homogeneous manifolds with uniformly positive scalar curvature, we can drop the assumption on the Yamabe invariant at infinity and reprove a result of Akutagawa:

\begin{thm}\label{isoYam}
Let $(M^n,g)$ be a manifold of bounded geometry, $\s\geq c>0$ for a constant $c$ and $Q(S^n)>Q(M,g)$. Furthermore, we assume that $(M,g)$ is almost homogeneous, i.e. there exists a relatively compact set $U\subset\subset M$ such that for all $x\in M$ there is an isometry $f:M\to M$ with $f(x)\in U$. 
Then, there is a positive smooth solution $v\in H_1^2\cap L^\infty$ of the Euler-Lagrange equation $Lv=Q(M)v^{p_{crit}-1}$ with $\Vert v\Vert_{p_{crit}}=1$.
\end{thm}

\begin{proof}
Due to the existence of the isometries, $M$ has bounded geometry. Moreover, since the scalar curvature is  uniformly positive, $\mu$ and $Q$ are positive. Hence, with Lemma  \ref{weisub}, we obtain positive solutions $v_{\alpha,p}\in H_1^2$ ($\alpha>0$, $p\in [2,p_{crit})$) of  $Lv_{\alpha,p}=Q^\alpha_p\rho^{\alpha p} v_{\alpha,p}^{p-1}$ with $\Vert \rho^\alpha v_{\alpha ,p}\Vert_p= 1$. Lemma \ref{conv} and \ref{Qinfmass} show that for a certain subsequence $v_p=v_{\alpha(p) ,p}$ converges to  $v$ in $C^2$-topology on each compact set. Moreover, $v\in H_1^2\cap L^\infty$ and $Lv \leq Q v^{p_{crit}-1}$. We need to show that $\Vert v\Vert_{p_{crit}}=1$: Due to Lemma \ref{max}, each $v_p$ has a maximum.  With the isometries, we can always pull the point $x_p$ where $v_{p}$ attains its maximum into the subset $U$.

Thus, without loss of generality we assume that $x_{p}\in U$. Since $v_p$ is maximal in $x_p$, we have that $\Delta v_p(x_p)\geq 0$ and, thus, $Qv_{p}^{p-2}(x_{p})\geq \s (x_{p}) \geq c$. Let $x\in \ov{U}$ be the limit of a convergent subsequence of $x_p$ as $p\to p_{crit}$.   Then $Q v^{p_{crit}-2}(x)\geq c >0$. Since $Q>0$ and $v$ is smooth, we have $0<\Vert v\Vert_{p_{crit}}$ and, thus, as in the proof of Lemma \ref{Qinfmass}, $\Vert v\Vert_{p_{crit}}=1$. Hence, we have a positive solution $v\in H_1^2$ of $Lv=Qv^{p-1}$ with $\Vert v\Vert_{p_{crit}}=1$.
\end{proof}

\begin{rem}
If there exist such isometries, as described in Theorem \ref{isoYam}, we have $\ov{Q(M)}=Q(M)$.\\
This can be seen when taking a minimizing sequence $v_i\in C_c^\infty(M)$ with $\Vert v_i\Vert_{p_{crit}}=1$ and $\int_M v_iLv_i\vo_g\to Q(M)$. Denote the diameter of $\mathrm{supp}\, v_i\cup U$ by $d_i$. Let $y\in U$ be fixed. We define $\tilde{v_i}=v_i\circ f_i$ where $f_i^{-1}$ is an isometry that a given point $x$ with $\mathrm{dist}(x, U)=i+d_i$ to a point in  $U$ Then, $\tilde{v_i}\in C_c^\infty(M\setminus B_i(y))$, $\int_M\tilde{v_i}L\tilde{v_i}\vo_g\to Q(M)$  and $\Vert \tilde{v_i}\Vert_{p_{crit}}=1$. Thus, $\ov{Q(M)}= Q(M)$.
\end{rem}

\begin{ex}\label{ex_ex}Consider the model spaces $(Z=S^{n-k-1}\times \mathbb{R}^{k+1},\, g_c=g_{S^{n-k-1}}+g_{c,k+1})$ which is a product of the standard sphere and the space $\mathbb{R}^{k+1}$ equipped with a metric of constant sectional curvature $-c^2k(k+1)$, $c\in [0,1]$. Those spaces appeared in \cite{ADH} and  have the symmetries required in the last remark. Their scalar curvature is constant and given by $\s_{g_c}=-k(k+1)c^2+(n-k-1)(n-k-2)$, e.g. for $k<\frac{n-2}{2}$ the scalar curvature is positive for all $c\in (0,1]$. Note that for $c=1$ $(Z,g_1)$ is conformal to $S^n\setminus S^k$ and thus $Q(Z,g_1)=Q(S^n)$.\\
Assuming that $c$ is chosen such that $\s_{g_c}$ is positive and $Q(Z,g_c)<Q(S^n)$, Theorem \ref{isoYam} shows that for those spaces there is a solution of the Euler-Lagrange equation. 
\end{ex}

Moreover, in \cite{ADH}, besides the Yamabe invariant from above the following invariant is used:
\[ \mu^{(1)}(M,g)=\inf\{ \mu\in \R\ |\ \exists u\in L^\infty\cap L^2, u\ne 0, \Vert u\Vert_{p_{crit}}\leq 1:\ L_gu=\mu u^{p_{crit}-1} \}.\]
The proof of \cite[Lem. 3.5]{ADH} shows, that if $(M,g)$ is a complete Riemannian manifold it is $\mu^{(1)}(M,g)\geq Q(M,g)$.

\begin{cor}
Let the assumptions of Theorem \ref{Yam} or of Theorem \ref{isoYam} be fulfilled for a manifold $(M,g)$. Then $\mu^{(1)}(M,g)=Q(M,g)$. 
\end{cor}

\begin{proof}
From Theorem  \ref{Yam} or \ref{isoYam} we know that there is a smooth solution $v\in H_1^2\cap L^\infty$ with $L_gv=Q v^{p_{crit}-1}$ and $\Vert v\Vert_{p_{crit}}=1$. Thus, $\mu^{(1)}\leq Q$. Hence, with $\mu^{(1)}(M,g)\geq Q(M,g)$ from above $Q(M,g)=\mu^{(1)}(M,g)$.
\end{proof}

\begin{appendix}\section{Embeddings on manifolds of bounded geometry}\label{app}

In \cite[Cor. 3.19]{Heb} there are already given continuous Sobolev embeddings for  manifolds of bounded geometry:

\begin{thm}\ \!\!\cite[Thm. 3.18 and Cor. 3.19]{Heb}
Let $(M^n,g)$ be a manifold of bounded geometry. Then $H_1^q(M)$ is continuously embedded in $L^p(M)$ for $\frac{1}{p}=\frac{1}{q}-\frac{1}{n}$.
\end{thm}

But unfortunately those embeddings are not compact. Therefore, we will work with weighted Sobolev embeddings:\\
Let $\rho: M\to (0,\infty)$ be a radial admissible weight, see \cite[Def. 2 and 4]{Skrz}. 

\begin{rem}\label{ch_weight}
In the following, we will choose 
$\rho(x)=\exp (-r)$ where $r$ is a smooth function with $ |r(x)-|x||<\xi$ for all $x\in M$ and a fixed $\xi>0$ where $|x|:=\mathrm{dist} (x,z)$ for fixed $z\in M$. On manifolds of bounded geometry, such a function $r$ always exists \cite[Lem. 2.1]{Shu}.
\end{rem}

We define the weighted $L^p$-space $\rho^\alpha L^p:=\{ f\ |\ \rho^\alpha f\in L^p (M)\}$ equipped with the norm $\Vert f\Vert_{\rho^\alpha L^p}:=\Vert \rho^\alpha f\Vert_{L^p}$ for $p\geq 1$.

\begin{thm}\ \!\!\label{embcon}\cite[Cor. 2]{Skrz}
If the manifold $(M,g)$ has bounded geometry, for each $2\leq p<p_{crit}=\frac{2n}{n-2}$ the Sobolev embedding $H_1^2 \hookrightarrow \rho^\alpha L^p$ is continuous for $\alpha\geq 0$ and compact  for $\alpha >0$.
\end{thm}

The hard part of the above theorem is to establish compactness. 
 
\begin{rem}\label{est}Let $(M,g)$ be a manifold with bounded geometry.\\
i) (Inner $L^p$-estimate)\cite[proof of Thm. 8.8 ]{GT}
Let $\epsilon \in (0, \frac{1}{2}\text{inj}(M))$ where $\text{inj}$ denotes the injectivity radius. Then there exists a constant $C_\epsilon(q)$ such that for all $x\in M$
\[ \Vert u\Vert_{H_2^q(B_\epsilon(x))}\leq C_\epsilon(q) (\Vert u\Vert_{L^q(B_{2\epsilon}(x))} +\Vert f\Vert_{L^q(B_{2\epsilon}(x))})\] 
for all $q\geq 1$, $f\in L_{loc}^q$ and where $u\in H_{2, loc}^q$ is a solution of $Lu=f$.\\ 
ii) (Imbedding) Let $n< q$ and $0\leq \gamma\leq 1-\frac{n}{q}$. From the proof of \cite[Sect. 7.8 (Thm 7.26)]{GT} we have that for all $\epsilon >0$ there exists a constant $C$ such that for all $x\in M$ the space 
$H_2^q(B_\epsilon(x))$ is continuously embedded in $ C^{0,\gamma}(\ov{B_\epsilon(x)})$  
\end{rem}

At the end we give a lemma which shows that solutions of the considered Euler-Lagrange equations have a maximum:

\begin{lem}\label{max} Let $(M,g)$ be a manifold of bounded geometry. Let $v\in H_1^2$ be a solution of $Lv=c\rho^{\alpha p} v^{p-1}$ with $\Vert \rho^\alpha v\Vert_p=1$.  
For $p< p_{crit}$, $v$ is continuous and $\lim_{|x|\to \infty} |v(x)|=0$.\\
Assume additionally that $v\in L^\infty$. Then we get the same also for $p=p_{crit}$.
\end{lem}

\begin{proof} Let $\epsilon\in (0, \frac{1}{2}\text{inj}(M))$. 
Assume that there exist a constant $V>0$ and a sequence $x_i\in M$ with $v(x_i)\geq V$ and $\text{dist}(x_i,p)\to \infty$ with $\text{dist}(x_i,x_j)>2\epsilon$ for fixed $p\in M$. We set $B_i=B_\epsilon(x_i)$, $B_{i,2}=B_{2\epsilon}(x_i)$. Then, the interior $L^p$-estimates from above give $\Vert v\Vert_{H_2^q(B_i)}\leq C_\epsilon(q) (\Vert v\Vert_{L^q(B_{i,2})} +\Vert \rho^{\alpha p} v^{p-1}\Vert_{L^q(B_{i,2})})$.\\ 
Moreover, the Sobolev embedding in Theorem \ref{embcon} shows that $v\in L^p$. From $\rho^\alpha v\in L^p$, $Lv=c\rho^{\alpha p} v^{p-1}$ and $0\leq \rho \leq 1$, we obtain $Lv \in L^{q_1}$ with $q_1=\frac{p}{p-1}$. The Schauder estimate above gives $v\in H_2^{q_1}(B_i)$ with $\Vert v\Vert_{H_2^{q_1}(B_i)}\leq CC_\epsilon$. Then the Sobolev embedding give $\Vert v\Vert_{L^{p_1}(B_i)}\leq C_\epsilon CC'$ with $p_1=\frac{nq_1}{n-q_1}$ and where $C'$ is the constant appearing in the corresponding Sobolev embedding. By a bootstrap argument we obtain a $q>n$ that $\Vert v\Vert_{H_2^q(B_i)}\leq K(q)$ where the constant $K(q)$ depends on $q$ but not on $i$. This bootstrap works since $p< p_{crit}$. Thus, with Remark \ref{est}.ii we get that $\Vert v\Vert_{C^{0,\alpha}(\ov{B_i})}\leq c_\alpha$ where $c_\alpha$ is independent of $i$ and $0< \alpha \leq 1-\frac{n}{q}$.\\
From Theorem \ref{embcon} we get from $v\in H_1^2$ that $v\in L^p$. Thus, \[\infty >\Vert v \Vert_p \geq \sum_i \Vert v\Vert_{L^p(B_\delta (x_i))}\geq  K \sum_i \min_{x\in B_\delta (x_i)} v(x) \] where $K^p=\inf \vol(B_\delta (x_i))$ and $\delta \leq \epsilon$. Thus, $\min_{x\in B_\delta (x_i)} v(x)\to 0$ as $i\to \infty$. But we know that on each $B_\delta (x_i)$ we have $|v(x)-v(y)|\leq c_\alpha|x-y|^\alpha\leq c_\alpha\delta^\alpha$. Thus in the limit for $i\to \infty$ we get $V\leq c_\alpha\delta^\alpha$. Choosing $\delta$ small enough we have a contradiction. Thus, $\limsup_{|x|\to \infty} v(x)=0$.

Let now $p=p_{crit}$ and $v\in L^\infty$. Then, together with a uniform upper bound for $\vol(B_{2,i})$, we can use directly Remark \ref{est}.i for a $q>n$  to obtain that $\Vert v \Vert_{H_2^q(B_i)}$ is uniformly bounded. Then the argument goes on as above. 
\end{proof}

\end{appendix}

\bibliographystyle{acm}
\bibliography{YamEL}

\end{document}